\documentclass[final,1p,times]{elsarticle}
\usepackage{amssymb}
\usepackage{amsmath}
\usepackage{amsthm}
\usepackage[utf8]{inputenc}
\usepackage{marginnote}
\usepackage{amsfonts}
\usepackage{textcomp}
\usepackage{color,soul}
\usepackage{placeins}
\usepackage{tikz}
\usepackage{makecell}
\usepackage{amssymb}
\usepackage{graphicx}
\usepackage{hyperref}
\usepackage[none]{hyphenat}
\usepackage{lipsum}
\newtheorem{thm}{Theorem}[section]
\newtheorem{lem}[thm]{Lemma}
\newtheorem{rem}[thm]{Remark}
\newtheorem{prop}[thm]{Proposition}
\newtheorem{cor}[thm]{Corollary}
\newtheorem{example}[thm]{Example}
\newdefinition{defn}[thm]{Definition}
\allowdisplaybreaks
\journal{}

\begin{document}

\begin{frontmatter}

\title{The structure of optimal dual frames for probabilistic erasures under  Hilbert--Schmidt norm }

\author[shan]{Shankhadeep Mondal}
\ead{shankhadeep.mondal@ucf.edu}
\author[ram]{R. N. Mohapatra}
\ead{ram.mohapatra@ucf.edu}
\author[yao]{Jen-Chih Yao}
\ead{yaojc@mail.cmu.edu.tw}

\address{Department of Mathematics, University of Central Florida}
\address{Department of Mathematics, University of Central Florida}
\address{Center for General Education, China Medical University, Taichung, Taiwan}

		\cortext[shan]{Corresponding Author: shankhadeep.mondal@ucf.edu}

\begin{abstract}
Frames provide redundant representations that enable stable signal reconstruction under coefficient losses. In this paper, we study optimal dual frames for probabilistic erasures using the Hilbert--Schmidt norm of the associated error operators. We characterize dual frames that are optimal for  $1-$erasures and establish conditions under which the canonical dual is not only optimal but also unique. We further derive lower bounds for the probabilistic reconstruction error for any $m-$ erasure and identify classes of frames for which the canonical dual remains optimal. In addition, we analyze the geometric structure of the set of optimal dual frames, showing that it is a nonempty compact convex set. These results provide new insights into robustness and optimal reconstruction in probabilistic erasure models.
\end{abstract}
\begin{keyword}
			Erasures, Frames, Codes, Optimal dual frame, Hilbert–Schmidt Norm.
			\MSC[2010] 42C15, 47B02, 94A12
		\end{keyword}

\end{frontmatter}

\section{Introduction}

Frames, introduced by Duffin and Schaeffer \cite{duff}, generalize orthonormal bases in Hilbert spaces by allowing redundant representations of vectors. This redundancy provides stability and robustness in signal reconstruction and has made frame theory an important tool in signal processing, coding theory, and data transmission. In practical communication systems, adverse network conditions such as congestion or channel limitations may cause loss of transmitted frame coefficients, leading to \emph{erasures}. The redundancy of frames enables recovery of signals even when some coefficients are missing, motivating the study of frame representations that are resilient to such losses.

A significant amount of research has focused on optimal reconstruction under erasures using frame-theoretic methods. Casazza and Kovačević \cite{casa2} investigated equal norm tight frames and analyzed their robustness to erasures. From a coding-theoretic viewpoint, Goyal, Kovačević, and Kelner \cite{goya} showed that uniform tight frames are optimal for handling single erasures. Holmes and Paulsen \cite{holm} introduced the operator norm as a measure of reconstruction error and derived conditions under which Parseval frames and their canonical duals minimize the worst-case error caused by erasures. Several subsequent works studied optimal dual frames under different error measures. Bodmann \cite{bodm1} considered the average of operator norms over all erasure patterns, while Lopez and Han \cite{jerr} obtained conditions for the canonical dual to be the unique optimal dual and analyzed the structure of optimal dual frames \cite{jins}. Pehlivan, Han, and Mohapatra \cite{sali} investigated optimal dual frames by minimizing the spectral radius of the associated error operators, and related problems for multiple erasures were further studied in \cite{peh,dev}. Alternative optimality measures such as the numerical radius and combined operator-norm measures have also been explored in \cite{ara,deep,shan1}. The study of optimal dual frames and optimal dual pairs has attracted considerable attention in recent years, and a variety of optimality criteria have been developed, including the operator norm, spectral radius, numerical radius, and averaged reconstruction-error measures. Related problems have also been investigated for several classes of frames and frame systems\cite{shan2,deep,deep1,ara,arab,arab1,arab2,shan5}.

A probabilistic framework for erasures was introduced in \cite{leng3}, where a probability distribution on erasure locations leads to an associated weight sequence used in the formulation of probabilistic error operators. This model has been further studied in \cite{leng,li2,shan,li1, shan1} in the context of optimal dual frames. These works demonstrate that the geometry of frame vectors and the choice of dual frames play a crucial role in minimizing reconstruction errors.

In this paper, we study optimal dual frames for probabilistic erasures using the Hilbert--Schmidt norm of the associated error operators as the measure of reconstruction error. We derive structural conditions under which the canonical dual frame becomes optimal and obtain quantitative bounds for the probabilistic reconstruction error in the presence of multiple erasures.

The paper is organized as follows. Section~2 reviews basic concepts of frame theory and introduces the probabilistic erasure model and associated error operators. Section~3 investigates probabilistic optimal dual frames and establishes conditions under which the canonical dual frame is optimal. In particular, we derive bounds for the $m$-erasure reconstruction error and characterize classes of frames for which the canonical dual remains optimal. Section~4 studies recovery of frame coefficients when erasure locations are unknown and proposes a probabilistic identification principle for determining the surviving coefficients.

\section{Preliminaries and the Probabilistic Erasure Model}

Let $\mathcal{H}_n$ denote an $n$-dimensional (real or complex) Hilbert space. 
A finite sequence of elements $F=\{f_i\}_{i=1}^N \subset \mathcal{H}_n$ is called a \textit{frame} for $\mathcal{H}_n$ if there exist constants $A,B>0$ such that
\[
A\|f\|^2 \leq \sum_{i=1}^N |\langle f,f_i\rangle|^2 \leq B\|f\|^2,
\qquad \forall f\in \mathcal{H}_n .
\]

The numbers $A$ and $B$ are called the \textit{frame bounds}. They are not unique. 
The \textit{optimal lower frame bound} is the supremum over all lower frame bounds, and the \textit{optimal upper frame bound} is the infimum over all upper frame bounds.

A frame $F=\{f_i\}_{i=1}^N$ is called \textit{normalized} if $\|f_i\|=1$ for every $i$. 
If $A=B$, that is,
\[
\sum_{i=1}^N |\langle f,f_i\rangle|^2 = A\|f\|^2, \qquad \forall f\in \mathcal{H}_n,
\]
then $F$ is called a \textit{tight frame}. If $A=B=1$, then $F$ is called a \textit{Parseval frame}. Every finite sequence $\{f_i\}_{i=1}^R$ in $\mathcal{H}_n$ is a frame for the Hilbert space $W := \mathrm{span}\{f_i\}_{i=1}^R .$

Let $F=\{f_k\}_{k=1}^N$ be a frame for $\mathcal{H}_n$. The linear mapping 
\[
T_F : \mathcal{H}_n \rightarrow \mathbb{C}^N, \qquad 
T_F(f)=\{\langle f,f_i\rangle\}_{i=1}^N
\]
is called the \textit{analysis operator}.  The adjoint operator $T_F^* : \mathbb{C}^N \rightarrow \mathcal{H}_n$ defined by
\[
T_F^*\left(\{c_i\}_{i=1}^N\right)=\sum_{i=1}^N c_i f_i
\]
is called the \textit{synthesis operator} (or \textit{preframe operator}). The \textit{frame operator} $S_F : \mathcal{H}_n \to \mathcal{H}_n$ is defined by
\[
S_F f = T_F^*T_F f = \sum_{i=1}^N \langle f,f_i\rangle f_i .
\]

The operator $S_F$ is positive and invertible on $\mathcal{H}_n$, which yields the reconstruction formula
\[
f=\sum_{i=1}^N \langle f,S_F^{-1}f_i\rangle f_i = \sum_{i=1}^N \langle f,f_i\rangle S_F^{-1}f_i.
\]

A frame $G=\{g_i\}_{i=1}^N$ in $\mathcal{H}_n$ is called a \textit{dual frame} of $F=\{f_i\}_{i=1}^N$ if every $f\in \mathcal{H}_n$ admits the representations
\[
f = \sum_{i=1}^N \langle f,f_i\rangle g_i
  = \sum_{i=1}^N \langle f,g_i\rangle f_i .
\]

It is known that $S_F^{-1}F=\{S_F^{-1}f_i\}_{i=1}^N$ is a frame for $\mathcal{H}_n$, called the \textit{ standard dual or canonical dual } of $F$. Moreover for $N>n$, there exist infinitely many dual frames of $F$ \cite{ole}. Every dual frame 
$G=\{g_i\}_{i=1}^N$ of $F$ can be written in the form
\[
g_i = S_F^{-1}f_i + u_i , \qquad i=1,\dots,N,
\]
where the sequence $\{u_i\}_{i=1}^N$ satisfies
\[
\sum_{i=1}^N \langle f,f_i\rangle u_i = 
\sum_{i=1}^N \langle f,u_i\rangle f_i =0,
\qquad \forall f\in \mathcal{H}_n .
\]

If $G$ is a dual frame of $F$, then
\[
\sum_{i=1}^N \langle g_i,f_i\rangle
= \mathrm{tr}(T_F T_G^*)
= \mathrm{tr}(T_G^* T_F)
= \mathrm{tr}(I)
= n .
\]
In particular, if $F$ is a Parseval frame, then $\sum_{i=1}^N \|f_i\|^2 = n .$

\medskip

During data transmission, it is possible that some frame coefficients in the reconstruction formula are lost. 
Let $\{p_i\}_{i=1}^N$ be a discrete probability distribution satisfying
\[
\sum_{i=1}^N p_i = 1, \qquad 0 \le p_i <1 .
\]

Here $p_i$ represents the probability that the $i$-th coefficient is erased. 
Following \cite{leng}, the associated weight number sequence $\{q_i\}_{i=1}^N$ is defined by
\[
q_i = \frac{1}{1-p_i}\cdot \frac{N-1}{n}, 
\qquad i=1,2,\dots,N .
\]

It is easy to verify that $\sum_{i=1}^N \frac{1}{q_i} = n .$ Moreover, if $N\ge2$ then $q_i>0$, and if $N>n$ then $1\le q_i<\infty$ for all $i$. If errors occur in $m$ positions, the probabilistic error operator with respect to the weights $\{q_i\}$ is defined by
\[
\mathcal{E}_{\Lambda,q}^{\text{prob}}(F,G)f
= T_G^* D^{(m)}_q T_F f
= \sum_{i\in\Lambda} q_i \langle f,f_i\rangle g_i ,
\qquad f\in \mathcal{H}_n ,
\]
where $\Lambda$ denotes the set of indices corresponding to the erased coefficients and $D^{(m)}_q$ is the diagonal matrix whose $i$-th diagonal entry equals $q_i$ if $i\in\Lambda$ and $0$ otherwise. For a frame $F$ and a dual $G$, the maximum probabilistic error for $m$ erasures is defined by
\[
\max \left\{ 
\mathcal{M}\big(T_G^* D_q^{(m)} T_F \big) :
D_q^{(m)} \in \mathcal{D}_q^{(m)}
\right\},
\]
where $\mathcal{M}$ denotes an appropriate measure of the error operator and $\mathcal{D}_q^{(m)}$ is the set of all diagonal matrices with exactly $m$ nonzero entries. The following notions are probabilistic analogues of the uniform dual frames introduced in the deterministic erasure setting.

\begin{defn}
Let $F=\{f_i\}_{i=1}^N$ be a frame for $\mathcal{H}_n$ and let $\{q_i\}_{i=1}^N$ be a weight number sequence. 
A dual frame $G=\{g_i\}_{i=1}^N$ of $F$ is called a \textit{probabilistic $1$-uniform dual} if there exists a constant $c$ such that $q_i \langle f_i,g_i\rangle = c,  \quad i=1,\dots,N .$
\end{defn}

\begin{rem}
If $G$ is a probabilistic $1$-uniform dual of $F$, then the constant must be $c=1$. 
Indeed,
\[
n = \sum_{i=1}^N \langle f_i,g_i\rangle
= \sum_{i=1}^N \frac{c}{q_i}
= c n .
\]
\end{rem}

\begin{defn}
Given weights $\{q_i\}_{i=1}^N$, a probabilistic $1$-uniform dual $G$ of $F$  is called a \textit{probabilistic $2$-uniform dual} if there exists a constant $c'$ such that
\[
q_i q_j \langle f_i,g_j\rangle \langle f_j,g_i\rangle = c' ,
\qquad i\neq j .
\]
\end{defn}

For each $k$, define
\[
\mathcal{F}_q^{(k)}(F,G)
= \max_{|\Lambda|=k}
\left\|\mathcal{E}_{\Lambda,q}^{\text{prob}}(F,G)\right\|_{HS},
\]
where $\|\cdot\|_{HS}$ denotes the Hilbert--Schmidt norm. 
Furthermore,
\[
\mathcal{F}_q^{(k)}(F)
= \inf \{\mathcal{F}_q^{(k)}(F,G) : G \text{ is a dual frame of } F\}.
\]

A dual frame $G$ of $F$ is called a \textit{$k$-erasure probabilistic optimal dual} if it is $(k-1)$-erasure probabilistic optimal and
\[
\mathcal{F}_q^{(k)}(F,G) = \mathcal{F}_q^{(k)}(F).
\]

\begin{prop}\label{prop2point4}
Let $F=\{f_i\}_{i=1}^N$ be a frame for $\mathcal{H}_n$ with weights $\{q_i\}_{i=1}^N$ and let $G=\{g_i\}_{i=1}^N$ be a dual frame of $F$. Then
\[
\mathcal{F}_q^{(1)}(F,G)
= \max_{1\le i\le N} q_i \|f_i\|\,\|g_i\|.
\]
\end{prop}

\begin{proof}
If an error occurs at index $i$, then
\[
\mathcal{E}_{\Lambda,q}^{\text{prob}}(F,G)f
= q_i\langle f,f_i\rangle g_i .
\]
A direct computation yields
\small \begin{align*}
    \left\|\mathcal{E}_{\Lambda,q}^{prob}(F,G) \right\|_{HS} = \left\|T_{G}^*D_q^{(1)}T_F\right\|_{HS} = \sqrt{tr\left(T_{F}^*D_q^{(1)}T_GT_{G}^*D_q^{(1)}T_F \right)} = \sqrt{tr\left(D_q^{(1)}T_FT_{F}^*D_q^{(1)}T_GT_{G}^* \right)}  = \sqrt{q_{i}^2\|f_i\|^2\|g_i\|^2} = q_i\|f_i\|\;\|g_i\|. 
\end{align*}
Hence
\[
\mathcal{F}_q^{(1)}(F,G)
= \max_{1\le i\le N} q_i \|f_i\|\,\|g_i\|.
\]
\end{proof}

The following proposition gives a lower bound for the quantity $\mathcal{F}_{q}^{(1)}(F,G)$ and characterizes the dual frames for which this bound is attained.

\begin{prop}\label{prop2point5}
Let $F=\{f_i\}_{i=1}^N$ be a frame for $\mathcal{H}_n$ with weight number sequence $\{q_i\}_{i=1}^N$, and let $G=\{g_i\}_{i=1}^N$ be a dual frame of $F$. Then
\(
\mathcal{F}_{q}^{(1)}(F,G) \geq 1 .
\)
Moreover, equality holds if and only if
\(
\|f_i\|\,\|g_i\| = \frac{1}{q_i}, \qquad i=1,2,\ldots,N .
\)
\end{prop}

\begin{proof}
Suppose $\mathcal{F}_{q}^{(1)}(F,G) < 1$. Then, $q_i \|f_i\|\,\|g_i\| < 1 ,$ for all $1 \leq i \leq N.$ Consequently, $\|f_i\|\,\|g_i\| < \frac{1}{q_i}, $ for all $ i .$ Using the fact that $G$ is a dual frame of $F$, we obtain
\[
n = \left| \sum_{i=1}^N \langle f_i , g_i \rangle \right|
\le \sum_{i=1}^N |\langle f_i , g_i \rangle|
\le \sum_{i=1}^N \|f_i\|\,\|g_i\|
< \sum_{i=1}^N \frac{1}{q_i}
= n,
\]
which is impossible. Hence $\mathcal{F}_{q}^{(1)}(F,G) \ge 1$. Now assume that $\mathcal{F}_{q}^{(1)}(F,G)=1$. Then $\max\limits_{1\le i\le N} q_i\|f_i\|\,\|g_i\| = 1 .$ If there exists $j$ such that $q_j\|f_j\|\,\|g_j\| < 1$, then $\|f_j\|\,\|g_j\| < \frac{1}{q_j},$
and the same argument as above yields
\[
n
= \left|\sum_{i=1}^N \langle f_i , g_i \rangle \right|
\le \sum_{i=1}^N \|f_i\|\,\|g_i\|
< \sum_{i=1}^N \frac{1}{q_i}
= n,
\]
a contradiction. Therefore, $q_i\|f_i\|\,\|g_i\| = 1, \quad \forall\, i .$

Conversely, if $\|f_i\|\,\|g_i\| = \frac{1}{q_i}$ for all $i$, then $q_i\|f_i\|\,\|g_i\| = 1, \quad \forall\, i,$ and hence
\(
\mathcal{F}_{q}^{(1)}(F,G) = \max\limits_{1\le i\le N} q_i\|f_i\|\,\|g_i\| = 1.
\)
This completes the proof.
\end{proof}

\begin{cor}
    Let $F = \{f_i\}_{i=1}^N$ be a frame for $\mathcal{H}_n$ with weight number sequence $\{q_i\}_{i=1}^N$  and $G = \{g_i\}_{i=1}^N$ be a dual of $F$ which satisfy $\mathcal{F}_{q}^{(1)}(F,G) = 1 .$ Then $G$ is a probabilistic $1-$uniform dual of $F.$
\end{cor}
\begin{proof}
 By Proposition \ref{prop2point5},  $\mathcal{F}_{q}^{(1)}(F,G) = 1 $  gives $\|f_i\|\,\|g_i\| = \frac{1}{q_i}, \quad \forall\, i.$ We then have
\[
n = \left|\sum_{i=1}^N \langle f_i , g_i \rangle \right| 
   \leq \sum_{i=1}^N \left| \langle f_i , g_i \rangle \right| 
   \leq \sum_{i=1}^N \|f_i\|\,\|g_i\|
   = \sum_{i=1}^N \frac{1}{q_i}
   = n.
\]
The chain of equalities implies that $|\langle f_j , g_j \rangle| = \|f_j\|\,\|g_j\|, \quad \forall\, 1 \leq j \leq N.$\\
Let $\langle f_j , g_j \rangle = a_j + i b_j$, where $a_j, b_j \in \mathbb{R}$ for each $j$.  
Then,
\[
\sum_{j=1}^N a_j = n, \qquad \sum_{j=1}^N b_j = 0, \qquad 
\sqrt{a_j^2 + b_j^2} = \frac{1}{q_j}, \quad 1 \leq j \leq N.
\]
Therefore,
\[
n = \sum_{j=1}^N a_j 
   \leq \sum_{j=1}^N \sqrt{a_j^2 + b_j^2}
   = \sum_{j=1}^N \frac{1}{q_j}
   = n.
\]
The equality condition forces $b_j = 0$ and $a_j = |a_j|$ for all $j$.  
Consequently,
\[
\langle f_j , g_j \rangle = \frac{1}{q_j}, \quad \forall\, 1 \leq j \leq N.
\]

\end{proof}

\section{Optimal Dual Frames for Probabilistic Erasures }

In this section, we investigate the problem of identifying dual frames that minimize the probabilistic reconstruction error under coefficient erasures. In particular, we study conditions under which the canonical dual frame $S_F^{-1}F$ achieves optimality with respect to the probabilistic $m$-erasure error. We first establish sufficient conditions guaranteeing that the canonical dual is the unique $1$-erasure probabilistic optimal dual. We then derive quantitative bounds for the $m$-erasure error and characterize classes of frames for which the canonical dual remains optimal for all erasure levels.

The following proposition provides a sufficient condition under which the canonical dual of \( F \) achieves probabilistic optimality. To this end, define
\[
M := \max \left\{ q_i\|f_i\| \cdot \|S_{F}^{-1} f_i\| : 1 \leq i \leq N \right\},
\]
and set
\[
I_1 := \left\{ i : q_i\|f_i\| \cdot \|S_{F}^{-1}f_i\| = M \right\}, 
\qquad 
I_2 := \{1,2,\ldots,N\} \setminus I_1.
\]
Let \( H_j := \mathrm{span}\{ f_i : i \in I_j \} \) for \( j = 1,2 \). The sufficient condition for optimality is expressed in terms of these subspaces.

\begin{prop}\label{prop4point1}
Let $F = \{f_i\}_{i=1}^N $ be a frame for $\mathcal{H}_n$ with weight number sequence $\{q_i\}_{i=1}^N$ . If $H_1 \cap H_2 = \{0\}$, then the canonical dual is a $1-$erasure probabilistic optimal dual of $F.$ 
\end{prop}
\begin{proof}
We shall give a proof by contradiction. Suppose there exists a dual $G = \{g_i\}_{i=1}^N $  such that  $\mathcal{F}_{q}^{(1)} (F,G) < \mathcal{F}_{q}^{(1)} (F, S_{F}^{-1} F) .$ Then for all $i \in I_1,\;	q_i\|f_i \|\,\|g_i \| < M =  q_i\|f_i \| \left\| S_{F}^{-1}f_i \right\| .$ The dual $G $ can be written as  $ \{ S_{F}^{-1}f_i + u_i \}_{i=1}^N ,$  where $\sum\limits_{i=1}^N \langle f,u_i \rangle f_i =0,$ for all $f \in \mathcal{H}_n.$
	So, we may write $\|S_{F}^{-1}f_i + u_i \|^2 < \|S_{F}^{-1}f_i \|^2,$ for all $i \in I_1,$ which in turn leads to
			\begin{align}\label{eqn4point9}
				\sum_{i \in I_1} \| u_i \|^2 + 2 Re \bigg( \sum_{i \in I_1} \langle  S_{F}^{-1}f_i ,u_i \rangle \bigg) < 0 .
			\end{align}
			Now, by the condition $\sum\limits_{i=1}^N  \langle f,u_i \rangle f_i =0,\,f \in \mathcal{H}_n$ and   $H_1 \cap H_2 = \{0\},$  we get  $\sum\limits_{i \in I_1} \langle  S_{F}^{-1}f_i ,u_i \rangle = 0.$  Thus, from \eqref{eqn4point9}, we have $ \sum\limits_{i \in I_1} \| u_i \|^2 < 0,$ which is absurd. Hence, $S_{F}^{-1}F$ is a $1-$erasure probabilistic optimal dual of $F.$
			
		\end{proof}

We now strengthen the previous result by identifying conditions under which the canonical dual is not only optimal but also unique.  The following proposition provides a sufficient condition for the uniqueness of the canonical dual as the $1$-erasure probabilistic optimal dual.

\begin{prop}\label{prop4point1}
Let \( F = \{f_i\}_{i=1}^N \) be a frame for the Hilbert space \( \mathcal{H}_n \) with weight number sequence $\{q_i\}_{i=1}^N$ . 
If \( H_1 \cap H_2 = \{0\} \) and \( \{f_i\}_{i \in I_2} \) is linearly independent,  then the canonical dual of \( F \) is the unique \( 1- \)erasure optimal dual of \( F \).
\end{prop}

\begin{proof}
Let \( G = \{g_i\}_{i=1}^N \) be a non-canonical dual of \( F \). Then \( G \) can be expressed as $g_i = S_F^{-1} f_i + u_i, \quad 1 \leq i \leq N,$
where the sequence \( \{u_i\}_{i=1}^N \) satisfies $\sum_{i=1}^N \langle f, u_i \rangle f_i = 0, \quad \forall\, f \in \mathcal{H}_n.$
Rewriting this condition in terms of the index sets \( I_1 \) and \( I_2 \), we obtain
\[
\sum_{i \in I_1} \langle f, u_i \rangle f_i
+ \sum_{i \in I_2} \langle f, u_i \rangle f_i
= 0, \qquad \forall\, f \in \mathcal{H}_n.
\]
Since \( H_1 \cap H_2 = \{0\} \), it follows that
\[
\sum_{i \in I_1} \langle f, u_i \rangle f_i = 0
\quad\text{and}\quad
\sum_{i \in I_2} \langle f, u_i \rangle f_i = 0,
\qquad \forall\, f \in \mathcal{H}_n.
\]
By the linear independence of \( \{f_i\}_{i \in I_2} \), we deduce that \( u_i = 0 \) for all \( i \in I_2 \). Next, consider \( U_1 = \{u_i\}_{i \in I_1} \) and \( F_1 = \{f_i\}_{i \in I_1} \). The orthogonality condition implies $T_{F_1}^* T_{U_1} = 0.$
Consequently,
\[
\sum_{i \in I_1} \langle S_F^{-1} f_i, u_i \rangle
= \operatorname{tr}\!\big( T_{U_1} T_{S_F^{-1} F_1}^* \big)
= \operatorname{tr}\!\big( T_{S_F^{-1} F_1}^* T_{U_1} \big)
= \operatorname{tr}\!\big( S_F^{-1} T_{F_1}^* T_{U_1} \big)
= 0.
\]
Thus,
\[
\operatorname{Re} \left( \sum_{i \in I_1} \langle S_F^{-1} f_i, u_i \rangle \right) = 0.
\]

Now consider the two possible cases:

Case 1:
If \( \operatorname{Re}\big( \langle S_F^{-1} f_i, u_i \rangle \big) = 0 \) for all \( i \in I_1 \), then
\small\[
\max_{1 \leq i \leq N} q_i\|f_i\| \, \|g_i\|
\geq
\max_{i \in I_1} q_i\|f_i\| \, \|g_i\|
=
\max_{i \in I_1}
q_i\|f_i\|
\sqrt{
\|S_F^{-1} f_i\|^2
+ \|u_i\|^2
+ 2 \operatorname{Re} \langle S_F^{-1} f_i, u_i \rangle
}
=
\max_{i \in I_1} \sqrt{ M^2 + q_{i}^2\|f_i\|^2 \|u_i\|^2 }
> M.
\]
\normalsize
Case 2:
If \( \operatorname{Re}\big( \langle S_F^{-1} f_i, u_i \rangle \big) \neq 0 \) for some \( i \in I_1 \), then there exists \( j \in I_1 \) such that
\( \operatorname{Re}\big( \langle S_F^{-1} f_j, u_j \rangle \big) > 0 \). Hence,
\[
\max_{1 \leq i \leq N} q_i\|f_i\| \, \|g_i\|
\geq
q_j\|f_j\|
\sqrt{
\|S_F^{-1} f_j\|^2
+ \|u_j\|^2
+ 2 \operatorname{Re} \langle S_F^{-1} f_j, u_j \rangle
}
> M.
\]

In either case, for any non-canonical dual \( G \), we obtain
\[
\mathcal{F}_{q}^{(1)}(F, G) > \mathcal{F}_{q}^{(1)}(F, S_F^{-1}F).
\]

Thus, the canonical dual \( S_F^{-1}F \) is the unique \( 1 \)-erasure probabilistic optimal dual of \( F \) and hence probabilistic optimal for all \( m \)-erasures.
\end{proof}

\begin{cor}\label{cor3point2}
Let \( F = \{f_i\}_{i=1}^N \) be a tight frame for the Hilbert space \( \mathcal{H}_n \) such that $q_i\|f_i\|^2  = c, \quad \text{for all } i,$
where \( c > 0 \) is a constant. Then the canonical dual frame \( S_F^{-1}F \) is the unique probabilistic optimal dual frame for any \( m \)-erasures.
\end{cor}

\begin{example}
Let $F=\{e_1,e_2,e_3,e_3\}\subset \mathbb{R}^3$ and take $p_1=p_2=0$, $p_3=p_4=\tfrac12$, so $q_1=q_2=1$, $q_3=q_4=2$. Then the canonical dual is $\{e_1,e_2,\tfrac12 e_3,\tfrac12 e_3\}$. Any dual has the form $g_1=e_1$, $g_2=e_2$, $g_3=u$, $g_4=e_3-u$. The objective is $\mathcal{F}_q^{(1)}(F,G)=\max\{1,1,2\|u\|,2\|e_3-u\|\}$, which is minimized iff $\|u\|\le \tfrac12$ and $\|e_3-u\|\le \tfrac12$. By the triangle inequality,
\[
1=\|e_3\|\le \|u\|+\|e_3-u\|\le \tfrac12+\tfrac12=1,
\]
so equality must hold. Hence $u$ and $e_3-u$ are positively collinear, and therefore $u=\lambda e_3$ for some $\lambda\in[0,1]$. The constraints give
\(
\lambda\le \tfrac12 \,\text{and}\, 1-\lambda\le \tfrac12,
\)
which imply $\lambda=\tfrac12$. Hence
\(
u=\tfrac12 e_3.
\)
Therefore,
\(
g_3=g_4=\tfrac12 e_3,
\)
and hence the canonical dual is the unique optimal dual.
\end{example}

The following proposition provides a sufficient condition for the canonical dual to be the unique $1$-erasure optimal dual.

\begin{prop}\label{prop:canonical-not-prob-opt}
Let $F=\{f_i\}_{i=1}^N$ be a frame for $H_n$ with weight number sequence
$\{q_i\}_{i=1}^N$. Assume that $\{f_i\}_{i\in I_1}$ is linearly independent, and that there exists a scalar sequence $\{c_i\}_{i=1}^N$ such that $\sum_{i=1}^N c_i f_i=0, \quad\text{and}\quad c_i\neq 0 \ \text{for all } i\in I_1.$ Then the canonical dual $S_F^{-1}F=\{S_F^{-1}f_i\}_{i=1}^N$ is not a $1$-erasure probabilistic optimal dual of $F$.
\end{prop}

\begin{proof}
Since $\{f_i\}_{i\in I_1}$ is linearly independent and $c_i\neq 0$ for all $i\in I_1$, the family $\{\bar{c_i} f_i\}_{i\in I_1}$ is linearly independent. Hence so is $\{S_F^{-1}(\bar{c_i} f_i)\}_{i\in I_1},$ as $S_F^{-1}$ is invertible. Therefore, there exists $h\in H_n$ such that $ \Re\langle S_F^{-1}(\bar{c_i} f_i),h\rangle<0,
\quad i\in I_1.$ Define $u_i:=c_i h,\quad i=1,\dots,N.$ Then for every $f\in H_n$,
\[
\sum_{i=1}^N \langle f,f_i\rangle u_i
=
\sum_{i=1}^N \langle f,f_i\rangle c_i h
=
\left\langle f,\sum_{i=1}^N c_i f_i\right\rangle h
=0.
\]
Hence, for every scalar $t\in\mathbb{R}$, $G_t:=\{S_F^{-1}f_i+t u_i\}_{i=1}^N$
is a dual frame of $F$.  Now for each $i\in I_1,$
\[
\|S_F^{-1}f_i+t u_i\|^2
=
\|S_F^{-1}f_i\|^2
+2t\,\Re\langle S_F^{-1}f_i,u_i\rangle
+t^2\|u_i\|^2.
\]
Since $u_i=c_i h$, we have $\Re\langle S_F^{-1}f_i,u_i\rangle = \Re\langle S_F^{-1}(\bar{c_i} f_i),h\rangle<0,\;$ for all $i \in I_1.$ Therefore, for all sufficiently small $t>0$, $\|S_F^{-1}f_i+t u_i\|<\|S_F^{-1}f_i\|,
\quad i\in I_1.$ This gives
\[
q_i\|f_i\|\,\|S_F^{-1}f_i+t u_i\|<M,
\qquad i\in I_1.
\]

Next, if $i\in I_2$, then by definition of $I_2$,
\(
q_i\|f_i\|\,\|S_F^{-1}f_i\|<M.
\)
By continuity, for sufficiently small $t>0$ we still have
\(
q_i\|f_i\|\,\|S_F^{-1}f_i+t u_i\|<M,
\quad i\in I_2.
\)
Combining the two cases, we obtain, for sufficiently small $t>0$,
\[
\max_{1\le i\le N} q_i\|f_i\|\,\|S_F^{-1}f_i+t u_i\|<M.
\]
Therefore by Proposition \ref{prop2point4},
\(
\mathcal{F}_q^{(1)}(F,G_t)< \mathcal{F}_q^{(1)}(F,S_F^{-1}F).
\)
Hence the canonical dual cannot be a $1$-erasure probabilistic optimal dual of $F$.
\end{proof}

The following proposition shows that $1$-erasure probabilistic optimality is invariant under unitary transformations.

\begin{prop}\label{thm4point3}
Let $F$ be a frame for $\mathcal{H}_n$ with weight number sequence $\{q_i\}_{i=1}^N$   and  $U$ be a unitary operator on $\mathcal{H}_n$. Then, $G$ is a $1-$erasure  optimal dual of $F$  if and only if $UG$  is a $1-$erasure  optimal dual of $UF.$
\end{prop}
			\begin{proof}
			Let G be a $1-$erasure probabilistic optimal dual of F with respect to the Hilbert–Schmidt norm. For any dual $G'$ of $UF,$ we know that $U^{*}G'$ is a dual of $F$ and so,
			\begin{align*}
				\mathcal{F}_{q}^{(1)} (UF,UG) = \mathcal{F}_{q}^{(1)} (F,G)  \leq \mathcal{F}_{q}^{(1)} (F, U^{*}G') = \mathcal{F}_{q}^{(1)} (UF,G').
			\end{align*}
			This shows that $UG $ is a $1-$erasure optimal dual of $UF.$  The converse can be proved by reversing the roles of $(F,G)$ and $(UF,UG)$ in the above argument and applying the action of the unitary operator $U^*$ on $(UF,UG)$.
		\end{proof}

\begin{rem}\label{remark3point5}
    Note that for any given frame $F,$ $\mathcal{F}_{q}^{(1)} (F) \geq 1.$
\end{rem}

The next result provides a quantitative lower bound for the probabilistic $m$-erasure error in terms of the weight sequence and the geometry of the frame. 

\begin{lem}\label{lemma3point6}
	Let $a_1,a_2,\ldots,a_k$ be $k$ real numbers (not necessarily distinct) and \\ 
	$ E= \left\{ (i_1,i_2,\ldots,i_\ell): i_1,i_2,\ldots,i_\ell \;\text{are all distinct and}\, 1 \leq i_j \leq k  \right\}.$ Then, $\max\limits_{(i_1,i_2,\ldots,i_\ell)\in E}\;\; a_{i_1}+ a_{i_2}+ \cdots + a_{i_\ell} \geq \frac{s\ell}{k},$ where $s = a_1+a_2+\cdots+a_k.$
\end{lem}

\begin{proof}
	Suppose $\max\limits_{(i_1,i_2,\ldots,i_\ell)\in E}\;\; a_{i_1}+ a_{i_2}+ \cdots + a_{i_\ell} < \frac{s\ell}{k}.$ Then, $a_{i_1}+ a_{i_2}+ \cdots + a_{i_\ell} < \frac{s\ell}{k},\,\forall \, \left(i_1, i_2,\ldots, i_\ell\right) \in E.$ Taking sum, we have $\sum\limits_{\left(i_1, i_2,\ldots, i_\ell\right) \in E} a_{i_1}+ a_{i_2}+ \cdots + a_{i_\ell} = {{k}\choose{\ell}}\frac{\ell}{k}\left( a_1 + a_2 + \cdots +a_k  \right) < {{k}\choose{\ell}}\frac{s\ell}{k},$ which is not possible as \; ${{k}\choose{\ell}}\frac{\ell}{k}\left( a_1 + a_2 + \cdots +a_k  \right) = {{k}\choose{\ell}}\frac{s\ell}{k}.$ Therefore, $\max\limits_{(i_1,i_2,\ldots,i_\ell)\in E}\;\; a_{i_1}+ a_{i_2}+ \cdots + a_{i_\ell} \geq \frac{s\ell}{k}.$
\end{proof}

\begin{thm}\label{thm3point5}
Let $F = \{f_i\}_{i=1}^N$ be a frame for $\mathcal{H}_n $ with weight number sequence $\{q_i\}_{i=1}^N$  such that it admits a dual $G$   satisfying $\|f_i\|\,\|g_i\| = \dfrac{1}{q_i},$ for all $i,$ then $\mathcal{F}_{q}^{(1)} (F) = 1$ and for $1 < m \leq N,\,$ 

$$\mathcal{F}_{q}^{(m)} (F) \geq 
				\sqrt{m + {{m}\choose{2}}\dfrac{n -\displaystyle{\sum_{i=1}^N \frac{1}{q_{i}^2}}}{\sum\limits_{i\neq j}\dfrac{1}{q_i q_j}}}, \;\;\text{when\; $n \geq \sum\limits_{i=1}^N \frac{1}{q_{i}^2}$}. $$

 Further, if the dual $G$ satisfy  $q_i q_j \langle g_i, g_j \rangle \langle f_j, f_i \rangle $ is a constant, for all $i \neq j,$ then
 $$\mathcal{F}_{q}^{(m)} (F,G) =
				\sqrt{m + {{m}\choose{2}}\dfrac{n -\displaystyle{\sum_{i=1}^N \frac{1}{q_{i}^2}}}{\sum\limits_{i\neq j}\dfrac{1}{q_i q_j}}}.$$
\end{thm}

\begin{proof}
It  is easy to see that $\mathcal{F}_{q}^{(1)} (F, G) = 1$ and by Proposition \ref{prop4point1} and by  Remark\ref{remark3point5} it follows that  $\mathcal{F}^{(1)}_{q} (F) = 1.$ For $ 1 <m \leq N,$ suppose error occurs in $i_1, i_2, \cdots ,i_m$ positions. Then,
  
  \begin{align} \label{eqn3point6}
     &\;\;\;\left\|\mathcal{E}_{\Lambda,q}^{prob}(F,G)\right\|_{HS} \nonumber\\&= \left\|T_{G}^*D_{q}^{(m)}T_F\right\|_{HS} \nonumber \\&=  \sqrt{tr(D_{q}^{(m)}T_FT_{F}^*D_{q}^{(m)}T_GT_{G}^*)}  \nonumber \\&= \sqrt{\sum_{r=1}^m q_{i_1}q_{i_r}\langle g_{i_1}, g_{i_r} \rangle \langle f_{i_r}, f_{i_1} \rangle + \sum_{r=1}^m q_{i_2}q_{i_r}\langle g_{i_2}, g_{i_r} \rangle \langle f_{i_r}, f_{i_2} \rangle + \cdots + \sum_{r=1}^m q_{i_m}q_{i_r}\langle g_{i_m}, g_{i_r} \rangle \langle f_{i_r}, f_{i_m} \rangle  } \nonumber\\&= \sqrt{\sum_{r=1}^m q_{i_r}^2\left\| g_{i_r} \right\|^2\left\| f_{i_r} \right\|^2 + \sum_{\substack{j,k=1 \\ j\neq k}}^m q_{i_j}q_{i_k}\langle g_{i_j}, g_{i_k} \rangle \langle f_{i_k}, f_{i_j} \rangle } \nonumber\\&= \sqrt{m + 2 Re \left( \sum_{\substack{j,k=1 \\ j> k }}^m q_{i_j}q_{i_k}\langle g_{i_j}, g_{i_k} \rangle \langle f_{i_k}, f_{i_j} \rangle  \right)} .
  \end{align}
  Therefore,
  \begin{align} \label{eqn3point7}
      \mathcal{F}_{q}^{(m)} (F, G) = \max\limits_{|\Lambda |=m} \left\|\mathcal{E}_{\Lambda,q}^{prob}(F,G)\right\|_{HS} = \sqrt{m + \max\limits_{(i_1,i_2,\ldots,i_m)\in \tilde{E}} 2Re\left( \sum_{\substack{j,k=1 \\ j> k \\ }}^m q_{i_j}q_{i_k} \langle g_{i_j}, g_{i_k} \rangle \langle f_{i_k}, f_{i_j} \rangle \right)},
  \end{align}
where $\tilde{E} = \left\{ (i_1,i_2,\ldots,i_m): i_1,i_2,\ldots,i_m \;\text{are all distinct and}\, 1 \leq i_j \leq N  \right\}.$
  Now, it can be easily seen that $\sum\limits_{i,j=1}^N \langle g_i, g_j \rangle \langle f_j, f_i \rangle = \sum\limits_{i=1}^N \left\langle g_i,\sum\limits_{j=1}^N \langle f_i, f_j \rangle g_j \right\rangle = \sum\limits_{i=1}^N \langle g_i ,f_i \rangle =n. $ Further, if $(F,G)$ satisfy $\|f_i\|\,\|g_i\| = \dfrac{1}{q_i}$, then
			\begin{align}\label{eqation7}
				\displaystyle{\sum_{i \neq j} \langle g_i, g_j \rangle \langle f_j, f_i \rangle} = n-\sum\limits_{i=1}^N \|f_i\|^2\,\|g_i\|^2 =n -\displaystyle{\sum_{i=1}^N \frac{1}{q_{i}^2}}.
			\end{align}

Let  $\beta := n -\displaystyle{\sum_{i=1}^N \frac{1}{q_{i}^2}} >0.$  Let us consider
$$  S_{m}(i_1,i_2,\ldots i_m) := \sum\limits_{j\neq k} q_{i_j} q_{i_k} \langle g_{i_j}, g_{i_k} \rangle \langle f_{i_k}, f_{i_j} \rangle .$$

Then by Lemma \ref{lemma3point6}, it is easy to see that 
$$\sum\limits_{(i_1,i_2,\ldots,i_m)\in \tilde{E}} S_{m}(i_1,i_2,\ldots i_m) = \dfrac{{{N}\choose{m}}{{m}\choose{2}}}{{{N}\choose{2}}} \sum\limits_{i\neq j} q_i q_j \langle g_i, g_j \rangle \langle g_j, g_i \rangle \geq \dfrac{{{N}\choose{m}}{{m}\choose{2}}}{{{N}\choose{2}}} q_{min}^2 \beta \geq {{N}\choose{m}}{{m}\choose{2}}  \dfrac{\beta}{\sum\limits_{i\neq j}{\dfrac{1}{q_i q_j}}}.$$

Therefore, $\max\limits_{(i_1,i_2,\ldots,i_m)\in \tilde{E}} S_{m}(i_1,i_2,\ldots i_m) \geq  {{m}\choose{2}}  \dfrac{\beta}{\sum\limits_{i\neq j}{\dfrac{1}{q_i q_j}}}.$ Thus by \eqref{eqn3point7}, we have 
$$ \mathcal{F}_{q}^{(m)} (F) \geq 
				\sqrt{m + {{m}\choose{2}}\dfrac{\beta}{\sum\limits_{i\neq j}\dfrac{1}{q_i q_j}}}$$

  Using the condition $\|f_i\|\,\|g_i\| = \frac{1}{q_i},\, 1\leq i \leq N,$ we have
  \begin{align}\label{equation3point8}
        2 Re \displaystyle{\left( \sum_{\substack{ i, j =1\\ i >j}}^N \langle g_i, g_j \rangle \langle f_j, f_i \rangle  \right)} = \sum\limits_{i \neq j}\langle g_i, g_j \rangle \langle f_j, f_i \rangle =n -\displaystyle{\sum_{i=1}^N \frac{1}{q_{i}^2}} = \beta.
  \end{align} \\

\par Now by assumption  $q_i q_j \langle g_i, g_j \rangle \langle f_j, f_i \rangle = c,$  for all $i \neq j$ and for some constant $c.$ Then, by \eqref{equation3point8},   $c= \frac{\beta}{\sum\limits_{i\neq j}\dfrac{1}{q_i q_j}}.$ Therefore, by \eqref{eqn3point7}, we have 
$$\mathcal{F}_{q}^{(m)} (F,G) =
				\sqrt{m + {{m}\choose{2}}\dfrac{\beta}{\sum\limits_{i\neq j}\dfrac{1}{q_i q_j}}}$$
 
\end{proof}

\begin{cor}
Let $F$ be a probabilistic uniform tight frame for $\mathcal{H}_n.$  If $\sqrt{q_i q_j}\left| \langle f_j, f_i \rangle \right| $ is a constant for all $i \neq j,$ then the canonical dual $S_{F}^{-1}F$ is a probabilistically optimal dual of $F$ for $m-$erasure for every $m \in \{1,2,\ldots,N\}$ . 
\end{cor}
\begin{proof}
Let $F$ be a tight frame with frame bound $A.$ Then,  $\|f_i\| = \sqrt{\dfrac{A}{q_i}},$ for all $i.$ Thus, $\mathcal{F}^{(1)}_q (F, S_{F}^{-1}F) = \max\limits_{1\leq i \leq N} \dfrac{q_i}{A} \|f_i\|^2 = 1.$ Then  by Proposition\ref{prop4point1}, $S_{F}^{-1}F$ is a $1-$erasure optimal dual frame of $F.$\\
As $F$ is a tight frame with frame bound $A,$ using duality of $F$ and $\frac{1}{A}F$, we have 
$$n=\sum\limits_{1 \leq i,j \leq N}\langle f_i,\frac{1}{A}f_j \rangle \langle \frac{1}{A}f_j,f_i \rangle = \sum\limits_{1 \leq i,j \leq N}\frac{1}{A^2}|\langle f_i,f_j \rangle|^2.$$
Using the probabilistic $1-$uniformity condition of the dual pair $(F,S_{F}^{-1}F),$ we have  $\sum\limits_{i \neq j} \frac{1}{A^2}|\langle f_i,f_j \rangle|^2 = n- \sum\limits_{i=1}^N\frac{1}{q_{i}^2}.$ Note that $n \geq \sum\limits_{i=1}^N\frac{1}{q_{i}^2}$,  otherwise it's not possible. Using the condition  $\sqrt{q_i q_j}\left| \langle f_j, f_i \rangle \right| $ is a constant for all $i \neq j,$ we have $\left|\langle f_i, f_j\rangle\right| = \dfrac{A\sqrt{n- \sum\limits_{i=1}^N\frac{1}{q_{i}^2}}}{\sum\limits_{i\neq j}\dfrac{1}{q_i q_j}}. $

Then by \eqref{eqn3point6}, we have $\mathcal{F}_{q}^{(m)} (F, S_{F}^{-1}F) = \sqrt{m + {{m}\choose{2}}\dfrac{\beta}{\sum\limits_{i\neq j}\dfrac{1}{q_i q_j}}},$ where $\beta = n- \sum\limits_{i=1}^N\frac{1}{q_{i}^2}.$ Then by Theorem\ref{thm3point5}, $ S_{F}^{-1}F$ is a $m-$erasure optimal dual of $F.$\\

\end{proof} 

We now examine the structural properties of the set of probabilistic $m$-erasure optimal dual frames. In particular, we show that this set enjoys strong geometric properties within the affine space of all dual frames. 

\begin{prop}\label{prop:topological-m-optimal}
Let $F=\{f_i\}_{i=1}^N$ be a frame for $\mathcal{H}_n$ with weight number sequence
$\{q_i\}_{i=1}^N$. For each $m\in\{1,\dots,N\}$,  the set of all $m$-erasure probabilistic optimal dual frames of $F$ is a nonempty compact convex subset of the affine space of all dual frames of $F$. In particular, it is closed and bounded.
\end{prop}

\begin{proof}
Let
\[
\mathcal D(F):=\{G=\{g_i\}_{i=1}^N:\ T_G^*T_F=I_{H_n}\}
\]
denote the affine space of all dual frames of $F$ and \(
\mathcal{OD}_q^{(m)}(F)\) denote the set of all $m$-erasure probabilistic optimal dual frames of $F.$ For each $m\in\{1,\dots,N\}$, define
\[
\Phi_m(G):=\mathcal F_q^{(m)}(F,G)
=\max_{|\Lambda|=m}\bigl\|T_G^*D_{q,\Lambda}T_F\bigr\|_{HS},
\qquad G\in\mathcal D(F),
\]
where $D_{q,\Lambda}$ is the diagonal matrix corresponding to the erasure set $\Lambda$.
By Definition 2.3, a dual frame $G$ is $m$-erasure probabilistic optimal if and only if
$G\in \mathcal{OD}_q^{(m-1)}(F)$ and
\[
\Phi_m(G)=\min\{\Phi_m(H): H\in \mathcal{OD}_q^{(m-1)}(F)\},
\]
with $\mathcal{OD}_q^{(0)}(F):=\mathcal D(F)$.

We prove the assertion by induction on $m$. For $m=1$, Proposition 2.4 gives
\[
\Phi_1(G)=\max_{1\le i\le N} q_i\|f_i\|\,\|g_i\|.
\]
Hence $\Phi_1$ is convex on $\mathcal D(F)$, since each map $G\mapsto q_i\|f_i\|\,\|g_i\|$ is convex and the maximum of finitely many convex functions is convex. Moreover, if
$G=\{g_i\}_{i=1}^N\in \mathcal{OD}_q^{(1)}(F)$, then
\[
q_i\|f_i\|\,\|g_i\|\le \Phi_1(G)=\inf_{H\in\mathcal D(F)}\Phi_1(H),
\qquad i=1,\dots,N,
\]
so each $\|g_i\|$ is uniformly bounded. Thus $\mathcal{OD}_q^{(1)}(F)$ is bounded.
Since $\Phi_1$ is continuous and $\mathcal D(F)$ is closed and affine, the minimizer set
$\mathcal{OD}_q^{(1)}(F)$ is closed and convex. In finite dimensions, closed and bounded implies compact. Therefore $\mathcal{OD}_q^{(1)}(F)$ is nonempty compact convex.

Assume now that $\mathcal{OD}_q^{(m-1)}(F)$ is nonempty compact convex for some $m\ge 2$.
We show the same for $\mathcal{OD}_q^{(m)}(F)$. First, $\Phi_m$ is continuous on $\mathcal D(F)$, because for each fixed $\Lambda$ the map $G\mapsto T_G^*D_{q,\Lambda}T_F$ is linear in $G$, and the Hilbert--Schmidt norm is continuous. Since the maximum is taken over finitely many $\Lambda$ with $|\Lambda|=m$, $\Phi_m$ is continuous.

Next, $\Phi_m$ is convex. Indeed, for each fixed $\Lambda,$ the map $ \; G\mapsto \bigl\|T_G^*D_{q,\Lambda}T_F\bigr\|_{HS}$, is convex, because it is the composition of a linear map with the Hilbert--Schmidt norm, which is convex. Therefore the maximum over all $|\Lambda|=m$ is also convex.

Since $\mathcal{OD}_q^{(m-1)}(F)$ is compact and nonempty, the continuous function $\Phi_m$
attains its minimum on $\mathcal{OD}_q^{(m-1)}(F)$. Thus $\mathcal{OD}_q^{(m)}(F)$ is nonempty.
Moreover,
\[
\mathcal{OD}_q^{(m)}(F)
=
\Bigl\{
G\in \mathcal{OD}_q^{(m-1)}(F):
\Phi_m(G)=\min_{H\in \mathcal{OD}_q^{(m-1)}(F)}\Phi_m(H)
\Bigr\}.
\]
Hence $\mathcal{OD}_q^{(m)}(F)$ is closed as a level set of a continuous function on the compact set
$\mathcal{OD}_q^{(m-1)}(F)$. Since $\Phi_m$ is convex and $\mathcal{OD}_q^{(m-1)}(F)$ is convex, the set of minimizers is convex. Finally, as a closed subset of the compact set $\mathcal{OD}_q^{(m-1)}(F)$, it is compact, and therefore bounded. Thus $\mathcal{OD}_q^{(m)}(F)$ is nonempty compact convex. By induction, the result holds for every
$m=1,\dots,N$.
\end{proof}

\begin{example}
Let $\mathcal{H}_3=\mathbb{R}^3$ and consider the frame
\[
F=\{f_1,f_2,f_3,f_4\}
=
\left\{
\begin{bmatrix}1\\0\\0\end{bmatrix},
\begin{bmatrix}0\\1\\0\end{bmatrix},
\begin{bmatrix}-2\\1\\1\end{bmatrix},
\begin{bmatrix}1\\-2\\-1\end{bmatrix}
\right\},
\]
with probability distribution 
\(
p_1=p_2=\frac16,\quad p_3=p_4=\frac13.
\)\;
Thus $q_1=q_2=\frac65,\quad q_3=q_4=\frac32.$

Since the excess of F is one and the unique linear dependence relation is $f_1+f_2+f_3+f_4=0,$ every dual frame of F has the form
$g_i=S_F^{-1}f_i+h,\; i=1,\dots,4,$
for some $h\in\mathbb R^3.$

Define the orthogonal operator $J:\mathbb{R}^3\to\mathbb{R}^3$ by
\[
J(x_1,x_2,x_3)=(x_2,x_1,-x_3).
\]
Then $Jf_1=f_2$, $Jf_2=f_1$, $Jf_3=f_4$, and $Jf_4=f_3$. Since $q_1=q_2$ and $q_3=q_4$, it follows that the objective
\[
\Phi(h):=\mathcal{F}_q^{(1)}(F,G(h))
\]
satisfies
\[
\Phi(Jh)=\Phi(h).
\]

Moreover, $\Phi$ is convex since it is the maximum of convex functions. Hence, for any $h\in\mathbb{R}^3$, the symmetrized vector
\[
\widetilde{h}=\frac{h+Jh}{2}
=\left(\frac{a+b}{2},\frac{a+b}{2},0\right)^T,\quad h=(a,b,c)^T,
\]
satisfies
\[
\Phi\left(\widetilde{h} \right)\le \Phi(h).
\]

\[
\Phi(\widetilde{h})
=
\max_{1\le i\le 4} q_i\|f_i\|\left\|S_F^{-1}f_i+\frac{h+Jh}{2}\right\|
=
\max_{1\le i\le 4} q_i\|f_i\|\left\|\frac{(S_F^{-1}f_i+h)+(S_F^{-1}f_i+Jh)}{2}\right\|.
\]
By the triangle inequality, we obtain
\[
\Phi(\widetilde{h})
\le
\max_{1\le i\le 4} \frac{q_i\|f_i\|}{2}
\Big(
\|S_F^{-1}f_i+h\|+\|S_F^{-1}f_i+Jh\|
\Big).
\]
Using the symmetry $Jf_1=f_2$, $Jf_2=f_1$, $Jf_3=f_4$, $Jf_4=f_3$ and $q_1=q_2$, $q_3=q_4$, this becomes
\[
\Phi(\widetilde{h})
\le
\frac12 \max\Big\{
A_1+A_2,\ A_2+A_1,\ A_3+A_4,\ A_4+A_3
\Big\},
\]
where $A_i=q_i\|f_i\|\,\|S_F^{-1}f_i+h\|$. Since $\frac{x+y}{2}\le \max\{x,y\}$ for $x,y\ge 0$, we conclude that
\[
\Phi(\widetilde{h})\le \max_{1\le i\le 4} A_i=\Phi(h).
\]

We justify the inequality $\Phi(\widetilde{h})\le \Phi(h)$. 
Recall that
\[
\Phi(h)=\max_{1\le i\le 4} q_i \|f_i\|\,\|S_F^{-1}f_i+h\|.
\]
For each fixed $i$, the function $h \mapsto \|S_F^{-1}f_i+h\|$
is convex, and hence $h \mapsto q_i\|f_i\|\,\|S_F^{-1}f_i+h\|$ is convex. Therefore $\Phi(h)$, being the maximum of finitely many convex functions, is convex.

Since $\widetilde{h}=\frac{h+Jh}{2}$ and $\Phi(Jh)=\Phi(h)$, convexity gives
\[
\Phi(\widetilde{h})
=
\Phi\!\left(\frac{h+Jh}{2}\right)
\le
\frac{\Phi(h)+\Phi(Jh)}{2}
=
\Phi(h).
\]

Thus, it suffices to consider $h=(a,a,0)^T$. Using Proposition~2.4, we obtain
\[
\mathcal{F}_q^{(1)}(F,G(a))
=
\max\left\{
\frac65\sqrt{2a^2+a+\frac{23}{8}},
\;
\frac32\sqrt{12a^2-6a+\frac94}
\right\}.
\]

At $a=0$ (the canonical dual), $\mathcal{F}_q^{(1)}(F,S_F^{-1}F)=\frac94.$

Both functions
\(
f_1(a)=\frac65\sqrt{2a^2+a+\frac{23}{8}},
\quad
f_2(a)=\frac32\sqrt{12a^2-6a+\frac94}
\)
are convex on the relevant interval. Moreover, \(f_1\) is increasing and \(f_2\) is decreasing near the point where the maximum is minimized. Hence
\[
\mathcal F_q^{(1)}(F,G(a))=\max\{f_1(a),f_2(a)\}
\]
attains its minimum at the unique intersection point of \(f_1\) and \(f_2\), namely
\[
a_0=\frac{83-\sqrt{4142}}{268}.
\]
Therefore this point gives the global minimum. Hence
\[
\mathcal{F}_q^{(1)}(F,G(a_0))<\mathcal{F}_q^{(1)}(F,S_F^{-1}F),
\]
so the canonical dual is not probabilistic optimal. Moreover, the minimizer is unique, and therefore $G(a_0)$ is the unique probabilistic $1$-erasure optimal dual of $F$.
\end{example}

\section{Recovery of Probabilistically Erased Frame Coefficients at Unknown Locations}
\label{sec:unknown_locations_prob}
\noindent The results of the previous sections characterize optimal dual frames for minimizing reconstruction errors arising from probabilistic erasures. These results implicitly assume that the locations of the erased coefficients are known. In many applications, however, the receiver has access only to the surviving coefficients and must first determine which coefficients have been lost. This motivates the study of erasure-location identification in the probabilistic setting. In this section, we develop an unknown-location recovery framework and establish geometric conditions under which the surviving coefficient set can be recovered uniquely for almost every signal, thereby complementing the optimal reconstruction theory developed earlier. In many transmission systems one may receive a subsequence of the frame coefficients
\[
\big\{\langle f,f_{i_1}\rangle,\ldots,\langle f,f_{i_{N-m}}\rangle\big\},
\qquad 1\le i_1<\cdots<i_{N-m}\le N,
\]
but the index set of the surviving coefficients (equivalently, the erasure set) is not known to the receiver.
In this section we formulate an ``unknown-location'' recovery principle in our \emph{probabilistic erasure} setting. We consider $m$ erasures at unknown locations.
For a subset $\Lambda\subset\{1,\dots,N\}$ with $|\Lambda|=m$ we denote the \emph{survival set} by
\[
I:=\Lambda^c,\qquad |I|=N-m,
\]
and write $I=\{i_1<\cdots<i_{N-m}\}$ when needed.

\subsection{Testing subspaces and probabilistic identifiability}

Recall that $T_F:\mathcal H_n\to\mathbb C^N$ is the analysis operator $T_F f=(\langle f,f_i\rangle)_{i=1}^N$.
For each $I=\{i_1<\cdots<i_{N-m}\}$ define the coordinate projection
$P_I:\mathbb C^N\to\mathbb C^{N-m}$ by
\[
P_I(c_1,\dots,c_N)=(c_{i_1},\dots,c_{i_{N-m}}),
\]
and define the corresponding (restricted) analysis map
\[
T_I := P_IT_F:\mathcal H_n\longrightarrow \mathbb C^{N-m}.
\]
Thus $T_I f$ is exactly the received coefficient vector when the surviving indices are $I$.

\begin{defn}
\label{def:prob_almost_robust}
Fix $m\in\{1,\dots,N-1\}$.
We say that $F=\{f_i\}_{i=1}^N$ is \emph{probabilistically almost robust with respect to $m$ erasures}  if there exists a set $\mathcal H_0\subset\mathcal H_n$ which is a finite union of
proper subspaces (hence of measure zero) such that for every $f\in\mathcal H_n\setminus\mathcal H_0$ the survival set $I$
(and hence the erasure locations $\Lambda$) is uniquely determined from the received vector $T_I f$ for all survival sets $I$
that occur with positive probability under the erasure law.
\end{defn}
This terminology reflects that the survival set can be identified uniquely for all signals outside a measure-zero exceptional set.
\noindent
In particular, if the erasure set $\Lambda$ is drawn from a distribution $\mathbb P$ supported on $\binom{N}{m}$,
then Definition~\ref{def:prob_almost_robust} requires identifiability for all $I$ with $\mathbb P(\Lambda^c=I)>0$.
(The weights $\{q_i\}$ enter the model through the probabilistic erasure mechanism and the error operator used elsewhere
in the paper; the unknown-location identifiability itself is a geometric property of the restricted analysis maps.)

\begin{thm}
\label{thm:prob_robust_characterization}
Let $F=\{f_i\}_{i=1}^N$ be a frame for $\mathcal H_n$, and fix $m$.
Let $\mathcal I_{m}:=\{I\subset\{1,\dots,N\}:\ |I|=N-m,\ \mathbb P(\Lambda^c=I)>0\}$ be the family of survival sets with
positive probability.
Then $F$ is probabilistically almost robust with respect to $m$ erasures if and only if the family of subspaces
\[
\big\{\,T_I(\mathcal H_n)\ :\ I\in\mathcal I_m\,\big\}
\]
consists of mutually different subspaces of $\mathbb C^{N-m}$.
\end{thm}

\begin{proof}

Assume $T_I(\mathcal H_n)\neq T_J(\mathcal H_n)$ for all distinct $I,J\in\mathcal I_m$.
Fix $I\in\mathcal I_m$.
For each $J\in\mathcal I_m$ with $J\neq I$, the intersection
\[
T_I(\mathcal H_n)\cap T_J(\mathcal H_n)
\]
is a proper subspace of $T_I(\mathcal H_n)$, hence its preimage under $T_I$ is a proper subspace of $\mathcal H$:
\[
\mathcal H_{I,J}:=T_I^{-1}\big(T_I(\mathcal H_n)\cap T_J(\mathcal H_n)\big)\subsetneq \mathcal H_n.
\]
Let $\mathcal H_0:=\bigcup_{I\in\mathcal I_m}\ \bigcup_{J\in\mathcal I_m,\ J\neq I}\ \mathcal H_{I,J}$,
a finite union of proper subspaces.
Now take any $f\in\mathcal H_n\setminus\mathcal H_0$ and suppose the received vector equals $x=T_I f$ for some
unknown $I\in\mathcal I_m$.
If $x\in T_J(\mathcal H_n)$ for some $J\neq I$, then $f\in \mathcal H_{I,J}\subset\mathcal H_0$, a contradiction.
Hence $x$ belongs to $T_I(\mathcal H_n)$ and to no other $T_J(\mathcal H)$ with $J\in\mathcal I_m$, $J\neq I$,
so $I$ is uniquely determined.

Conversely, if $T_I(\mathcal H_n)=T_J(\mathcal H_n)$ for some distinct $I,J\in\mathcal I_m$, then for every $f\in\mathcal H_n$
the received vector $T_I f$ lies in $T_J(\mathcal H_n)$ as well, so the survival set cannot be uniquely identified from
$T_I f$ on any set of full measure. Thus probabilistic almost robustness fails.
\end{proof}

\subsection{A generic scaling result in the probabilistic setting}

A key feature of the unknown-location model is that identifiability may fail for certain ``degenerate'' frames but can be
restored by a suitable rescaling, analogous to the deterministic unknown-location setting.

\begin{thm}
\label{thm:prob_scaling}
Let $F=\{f_i\}_{i=1}^N$ be a frame for $\mathcal H_n$ with uniform excess $k$ (so $N=n+k$), and fix $m<k$.
Assume the erasure law satisfies $\mathbb P(\Lambda^c=I)>0$ for all (or at least for some family of) $|I|=N-m$ survival sets.
Then there exist nonzero scalars $s_1,\dots,s_N$ such that the scaled frame
\[
\widetilde F=\{\,\widetilde f_i:=s_i f_i\,\}_{i=1}^N
\]
is probabilistically almost robust with respect to $m$ erasures (for the same erasure law).
\end{thm}

\begin{proof}
Consider the mapping $T_I=P_IT_{\widetilde F}$ corresponding to the scaled frame.
Each subspace $T_I(\mathcal H_n)\subset\mathbb C^{N-m}$ depends algebraically on the scaling vector
$s=(s_1,\dots,s_N)\in(\mathbb C^\ast)^N$ through the entries of $T_{\widetilde F}$.
For fixed distinct $I,J$, the condition $T_I(\mathcal H_n)=T_J(\mathcal H_n)$ imposes a set of polynomial constraints on $s$
(equivalently, a proper algebraic variety), because equality of the ranges forces all $(n\times n)$ minors defining the
Grassmannian coordinates of these subspaces to match.
Since $m<k$ (so $N-m>n$), the ranges $T_I(\mathcal H_n)$ have dimension $n$ inside $\mathbb C^{N-m}$, and the set of scalings
for which two such ranges coincide is a proper (hence measure-zero) subset of $(\mathbb C^\ast)^N$.
Taking the finite union over all distinct pairs $I\neq J$ in $\mathcal I_m$ still yields a measure-zero exceptional set.
Choose $s$ outside this union; then $T_I(\mathcal H_n)\neq T_J(\mathcal H_n)$ for all distinct $I,J\in\mathcal I_m$.
The conclusion follows from Theorem~\ref{thm:prob_robust_characterization}.
\end{proof}

\subsection{A recovery principle and an implementable test}

Once the survival set $I$ is identified, the original signal $f$ can be reconstructed whenever the surviving subsystem
$\{\widetilde f_i\}_{i\in I}$ admits exact reconstruction. In that case, using a suitable dual frame
$\widetilde G=\{\widetilde g_i\}_{i=1}^N$ associated with the surviving coefficients, one has
\[
f=\sum_{i\in I}\langle f,\widetilde f_i\rangle\,\widetilde g_i.
\]
Therefore, the central task is to identify $I$ from the received coefficient vector.
Theorem~\ref{thm:prob_robust_characterization} is geometric; to turn it into a computational test, one can proceed as follows.

For each $I=\{i_1<\cdots<i_{N-m}\}$, choose a full row-rank matrix $M_I\in\mathbb C^{(N-m-n)\times (N-m)}$
whose nullspace equals $T_I(\mathcal H)$:
\[
\mathcal N(M_I)=T_I(\mathcal H_n).
\]
(Equivalently, $M_I$ can be obtained by computing a basis of the orthogonal complement of $T_I(\mathcal H_n)$ in
$\mathbb C^{N-m}$.) Then, for noise-free data $x=T_I f$ we have $M_I x=0$, while $M_J x\neq 0$ for all $J\neq I$ for
generic $f$ (outside $\mathcal H_0$). This yields an identification rule:
\[
I^\ast=\arg\min_{I\in\mathcal I_m}\ \|M_I x\|.
\]
In the probabilistic setting, it is natural to incorporate the weights $\{q_i\}$ (or equivalently $\{p_i\}$) by restricting
the search to the most likely survival sets (or by adding a likelihood penalty).
For example, one can select
\[
I^\ast=\arg\min_{I\in\mathcal I_m}\ \big(\|M_I x\|^2-\lambda\log\mathbb P(\Lambda^c=I)\big),
\]
with $\lambda>0$ tuned to the anticipated noise level.

\begin{prop}
\label{thm:prob_identification}
Assume $F$ is probabilistically almost robust with respect to $m$ erasures and fix $I\in\mathcal I_m$.
Let $M_I$ be any matrix satisfying $\mathcal N(M_I)=T_I(\mathcal H_n)$.
Then there exists a measure-zero set $\mathcal H_0\subset\mathcal H_n$ such that for every $f\in\mathcal H_n\setminus\mathcal H_0$
the received vector $x=T_I f$ satisfies
\[
\|M_I x\|=0
\quad\text{and}\quad
\|M_J x\|>0\ \ \text{for all }J\in\mathcal I_m,\ J\neq I.
\]
In particular, the minimization rule $I^\ast=\arg\min_{J\in\mathcal I_m}\|M_J x\|$ recovers the true survival set $I$
with probability $1$ under the erasure law (and with respect to $f$ drawn from any distribution absolutely continuous
with respect to Lebesgue measure on $\mathcal H_n$).
\end{prop}

\begin{proof}
Since $\mathcal N(M_I)=T_I(\mathcal H_n)$, we have $M_I x=0$ for all $x\in T_I(\mathcal H)$, hence $\|M_I x\|=0$.
For $J\neq I$, probabilistic almost robustness implies $T_I(\mathcal H_n)\neq T_J(\mathcal H_n)$, so
$T_I(\mathcal H_n)\cap T_J(\mathcal H_n)$ is a proper subspace of $T_I(\mathcal H_n)$.
Let $\mathcal H_{I,J}:=T_I^{-1}(T_I(\mathcal H_n)\cap T_J(\mathcal H_n))$, a proper subspace of $\mathcal H_n$.
With $\mathcal H_0=\bigcup_{J\neq I}\mathcal H_{I,J}$ (finite union), for $f\notin\mathcal H_0$ we have
$x=T_I f\notin T_J(\mathcal H_n)$, hence $M_J x\neq 0$ and $\|M_J x\|>0$.
The probabilistic statement follows because the erasure law chooses $I\in\mathcal I_m$ with positive probability and the
exceptional set $\mathcal H_0$ has measure zero.
\end{proof}

\medskip
\noindent\textbf{Remark.}
Theorems~\ref{thm:prob_robust_characterization} and \ref{thm:prob_scaling} are the probabilistic analogues of the
unknown-location robustness results in the deterministic setting:
the probability model enters through the support $\mathcal I_m$ and (optionally) through likelihood-weighted testing,
while the identifiability mechanism is governed by the geometry of the restricted analysis ranges.
\\~\\

\noindent
	{\bf Acknowledgments:} 
	The authors are grateful to the Mohapatra Family Foundation and the College of Graduate Studies of the University of Central Florida for their support during this research.
	
	\section{Declaration} 
\noindent Ethical approval: Not applicable.\\
Data Availability Statement: Not Applicable.\\
Conflict of Interest: Not Applicable.\\
Clinical Trial: Not applicable.

\section{Funding}

This research did not receive any specific grant from funding agencies in the public, commercial, or not-for-profit sectors.

\bibliographystyle{amsplain}

\end{document}